\documentclass[a4paper]{amsart}
\usepackage{amsmath,amsthm,amssymb,latexsym,epic,bbm,comment}
\usepackage{graphicx,enumerate,stmaryrd}
\usepackage[all,2cell]{xy}
\xyoption{2cell}

\newtheorem{theorem}{Theorem}
\newtheorem{lemma}[theorem]{Lemma}

\newtheorem{proposition}[theorem]{Proposition}

\usepackage[all]{xy}
\usepackage[active]{srcltx}
\usepackage[parfill]{parskip}
\usepackage{enumerate}
\numberwithin{equation}{section}

\font\sc=rsfs10

\newcommand{\cC}{\sc\mbox{C}\hspace{1.0pt}}

\newcommand{\cS}{\sc\mbox{S}\hspace{1.0pt}}

\newcommand{\cA}{\sc\mbox{A}\hspace{1.0pt}}
\newcommand{\cB}{\sc\mbox{B}\hspace{1.0pt}}

\font\scc=rsfs7
\newcommand{\ccC}{\scc\mbox{C}\hspace{1.0pt}}

\newcommand{\ccA}{\scc\mbox{A}\hspace{1.0pt}}
\newcommand{\ccB}{\scc\mbox{B}\hspace{1.0pt}}

\newcommand{\ccS}{\scc\mbox{S}\hspace{1.0pt}}

\font\sccc=rsfs5
\newcommand{\cccA}{\sccc\mbox{A}\hspace{1.0pt}}

\begin{document}

\title[Fiat categorification of $IS_n$ and $F^*_n$]
{Fiat categorification of the symmetric inverse semigroup $IS_n$ and the semigroup $F^*_n$}
\author{Paul Martin and Volodymyr Mazorchuk}

\begin{abstract}
Starting from the symmetric group $S_n$, we construct two
fiat $2$-categories. One of them can be viewed as the fiat
``extension'' of the natural $2$-category associated with the 
symmetric inverse semigroup (considered as an ordered semigroup
with respect to the natural order). This $2$-category provides
a fiat categorification for the integral semigroup algebra of 
the symmetric inverse semigroup. The other $2$-category can be viewed 
as the fiat ``extension'' of the $2$-category associated with the 
maximal factorizable subsemigroup of the dual symmetric inverse 
semigroup (again, considered as an ordered semigroup
with respect to the natural order). This $2$-category provides
a fiat categorification for the integral semigroup algebra of 
the maximal factorizable subsemigroup of the dual symmetric inverse 
semigroup. 
\end{abstract}
\maketitle

\section{Introduction and description of the results}\label{s1}

Abstract higher representation theory has its origins in the papers \cite{BFK,CR,Ro1,Ro2}
with principal motivation coming from \cite{Kh,St}.
For finitary $2$-categories, basics of $2$-representation theory were developed in \cite{MM1,MM2,MM3,MM4,MM5,MM6}
and further investigated in \cite{GrMa1,GrMa,Xa,Zh1,Zh2,Zi,CM,MZ,MaMa,KMMZ}, 
see also \cite{KiMa} for applications. For different ideas on higher representation theory, see also
\cite{ForresterBarker,Baez,Elgueta,Pfeiffer,Kapranov} and references therein.

The major emphasis in \cite{MM1,MM2,MM3,MM4,MM5,MM6} is on the study of so-called {\em fiat}
$2$-categories, which are $2$-categorical analogues of finite dimensional algebras with involution.
Fiat $2$-categories appear naturally both in topology and representation theory.
They have many nice properties and the series of papers mentioned above develops an
essential starting part of $2$-representation theory for fiat categories.

Many examples of $2$-categories appear naturally in semigroup theory, see \cite{KuMa,GrMa1,GrMa,Fo}.
The easiest example is the $2$-category associated to a monoid with a fixed admissible partial order,
see Subsection~\ref{s4.0} for details. Linear analogues of these $2$-categories show up
naturally in representation theory, see \cite{GrMa1,GrMa}. A classical example of an ordered monoid
is an inverse monoid with respect to the natural partial order. There is a standard linearization procedure,
which allows one to turn a $2$-category of a finite ordered monoid into a finitary $2$-category,
see Subsection~\ref{s3.2} for details.

One serious disadvantage with linearizations of $2$-categories associated to finite ordered monoids is the
fact that they are almost never fiat. The main reason for that is lack of $2$-morphisms which start from the
identity $1$-morphism. In the present paper we construct two natural ``extensions'' of the symmetric group
to $2$-categories whose linearizations are fiat. One of them becomes a nice $2$-categorical analogue 
(categorification) for the symmetric inverse semigroup $IS_n$. The other one becomes a nice $2$-categorical 
analogue for the maximal factorizable subsemigroup  $F^*_n$ in the dual symmetric inverse semigroup $I^*_n$.

The main novel component of the present paper is in the definitions and constructions of the
main objects. To construct our $2$-categories, we, essentially, have to define three things:
\begin{itemize}
\item sets of $2$-morphisms between elements of $S_n$;
\item horizontal composition of $2$-morphisms;
\item vertical composition of $2$-morphisms.
\end{itemize}
In the case which eventually leads to $IS_n$, we view elements of $S_n$ as binary relations in the
obvious way and define $2$-morphisms between two elements of $S_n$ as the set of all binary relations
contained in both these elements. We chose vertical composition to be given by intersection of
relations and horizontal composition to be given by the usual composition of relations. Although all these
choices are rather natural, none of them seems to be totally obvious. Verification that this indeed
defines a $2$-category requires some technical work. In the case which eventually leads to $F^*_n$,
we do a similar thing, but instead of binary relations, we realize $S_n$ inside the partition monoid.
For $2$-morphisms between elements $\sigma$ and $\tau$ in $S_n$, we use those partitions which contain 
both $\sigma$ and $\tau$. All details on both constructions and all verifications can be found in Section~\ref{s2}.

Section~\ref{s3} recalls the theory of $\Bbbk$-linear $2$-categories  and gives
explicit constructions for a finitary $\Bbbk$-linear $2$-category starting
from a finite $2$-category. In Section~\ref{s4} we establish that our constructions
lead to fiat $2$-categories. We also recall, in more details, the standard constructions
of finitary $2$-categories, starting from $IS_n$ and $F^*_n$, considered as ordered monoids,
and show that the $2$-categories obtained in this way are not fiat. In Section~\ref{s5}
we make the relation between our constructions and $IS_n$ and $F^*_n$ precise. In fact, 
we show that the decategorification of our first construction is isomorphic to the semigroup algebra
$\mathbb{Z}[IS_n]$, with respect to the so-called M{\"o}bius basis in $\mathbb{Z}[IS_n]$, cf. \cite[Theorem~4.4]{Ste}.
Similarly, we show that the decategorification of our second construction is isomorphic to the semigroup algebra
$\mathbb{Z}[F^*_n]$, with respect to a similarly defined basis. We complete the paper with two explicit
examples in Section~\ref{s6}.
\vspace{5mm}

{\bf Acknowledgment.} The main part of this research was done during thew visit of the second author
to University of Leeds in October 2014. Financial support of EPSRC and hospitality of University of Leeds
are gratefully acknowledged. The first author is partially supported by EPSRC under grant EP/I038683/1.
The second author is partially supported by the Swedish Research Council and 
G{\"o}ran Gustafsson Foundation. We thank Stuart Margolis for stimulating discussions.

\section{Two $2$-categorical ``extensions'' of $S_n$}\label{s2}

\subsection{$2$-categories}\label{s2.1}
A {\em $2$-category} is a category enriched over the monoidal category $\mathbf{Cat}$ of small categories. This
means that a $2$-category $\cC$ consists of 
\begin{itemize}
\item objects $\mathtt{i},\mathtt{j},\dots$;
\item small morphism categories $\cC(\mathtt{i},\mathtt{j})$;
\item bifunctorial compositions;
\item identity objects $\mathbbm{1}_{\mathtt{i}}\in \cC(\mathtt{i},\mathtt{i})$;
\end{itemize}
which satisfy the obvious collection of (strict) axioms. Objects in morphism categories are usually called
{\em $1$-morphisms} (for example, all $\mathbbm{1}_{\mathtt{i}}$ are $1$-morphisms)
while morphisms in morphism categories are usually called {\em $2$-morphisms}.
Composition of  $2$-morphisms inside a fixed $\cC(\mathtt{i},\mathtt{j})$ is called {\em vertical}
and denoted $\circ_1$. Composition of  $2$-morphisms coming from the bifunctorial composition in 
$\cC$ is called {\em horizontal} and denoted $\circ_0$.
We refer the reader to \cite{Mac,Le} for more details on $2$-categories.

The main example of a $2$-category is  $\mathbf{Cat}$ itself, where
\begin{itemize}
\item objects are small categories;
\item $1$-morphisms are functors;
\item $2$-morphisms are natural transformations;
\item composition is the usual composition;
\item identity $1$-morphisms are identity functors.
\end{itemize}

\subsection{First $2$-category extending $S_n$}\label{s2.3}

For $n\in\mathbb{N}:=\{1,2,3,\dots\}$, consider the set $\mathbf{n}=\{1,2,\dots,n\}$ and let $S_n$ denote the 
{\em symmetric group} of all bijective transformations of $\mathbf{n}$ under composition. We consider also
the monoid $\mathbf{B}_n=2^{\mathbf{n}\times \mathbf{n}}$ of all {\em binary relations} on $\mathbf{n}$ which
is identified with the monoid of {\em $n\times n$-matrices} over the Boolean semiring $\mathbf{B}:=\{0,1\}$
by taking a relation to its adjacency matrix. Note that $\mathbf{B}_n$ is an ordered monoid with 
respect to usual inclusions of binary
relations. We identify $S_n$ with the group of invertible elements in $\mathbf{B}_n$ in the obvious way.

We now define a $2$-category $\cA=\cA_n$. To start with, we declare that 
\begin{itemize}
\item $\cA$ has one object $\mathtt{i}$;
\item $1$-morphisms in $\cA$ are elements in $S_n$;
\item composition $\cdot$ of $1$-morphisms is induced from $S_n$;
\item the identity $1$-morphism is the identity transformation $\mathrm{id}_{\mathbf{n}}\in S_n$.
\end{itemize}
It remains to define $2$-morphisms in $\cA$ and their compositions.
\begin{itemize}
\item For $\pi,\sigma\in S_n$, we define $\mathrm{Hom}_{\ccA}(\pi,\sigma)$ 
as the set of all $\alpha\in \mathbf{B}_n$ such that $\alpha\subseteq \pi\cap \sigma$.
\item For $\pi,\sigma,\tau\in S_n$, and also for $\alpha\in\mathrm{Hom}_{\ccA}(\pi,\sigma)$ and
$\beta\in \mathrm{Hom}_{\ccA}(\sigma,\tau)$, we define $\beta\circ_1\alpha:=\beta\cap \alpha$.
\item For $\pi\in S_n$, we define the identity element in 
$\mathrm{Hom}_{\ccA}(\pi,\pi)$ to be $\pi$.
\item For $\pi,\sigma,\tau,\rho\in S_n$, and also for $\alpha\in\mathrm{Hom}_{\ccA}(\pi,\sigma)$ and
$\beta\in\mathrm{Hom}_{\ccA}(\tau,\rho)$, we define $\beta\circ_0\alpha:=\beta\alpha$, the
usual composition of binary relations.
\end{itemize}

\begin{proposition}\label{prop1}
The construct $\cA$ above is a $2$-category.
\end{proposition}

\begin{proof}
Composition $\cdot$ of $1$-morphisms is associative as $S_n$ is a group.
The vertical composition  $\circ_1$  is clearly well-defined. It is associative as $\cap$ is associative.
If we have $\alpha\in\mathrm{Hom}_{\ccA}(\pi,\sigma)$ or $\alpha\in\mathrm{Hom}_{\ccA}(\sigma,\pi)$, then 
$\alpha\subseteq \pi$ and thus $\alpha\cap\pi =\alpha$. Therefore $\pi\in \mathrm{Hom}_{\ccA}(\pi,\pi)$ 
is the identity element.
 
Let us check that the horizontal composition $\circ_0$ is well-defined. From $\alpha\subseteq \pi$ and 
$\beta\subseteq\tau$ and the fact that $\mathbf{B}_n$ is ordered, we have $\beta\alpha\subseteq \tau\alpha
\subseteq \tau\pi$. Similarly, from $\alpha\subseteq \sigma$ and 
$\beta\subseteq\rho$ and the fact that $\mathbf{B}_n$ is ordered, we have $\beta\alpha\subseteq \rho\alpha
\subseteq \rho\sigma$. It follows that $\beta\alpha\in \mathrm{Hom}_{\ccA}(\tau\pi,\rho\sigma)$ and thus
$\circ_0$   is well-defined. Its associativity follows from the fact that usual composition of binary relations
is associative.

It remains to check the {\em interchange law}, that is the fact that, for any $1$-morphisms
$\pi,$ $\sigma,$ $\rho,$ $\tau,$ $\mu,$ $\nu$ and for any 
$\alpha\in \mathrm{Hom}_{\ccA}(\pi,\sigma)$, $\beta\in \mathrm{Hom}_{\ccA}(\tau,\mu)$,
$\gamma\in \mathrm{Hom}_{\ccA}(\sigma,\rho)$ and $\delta\in \mathrm{Hom}_{\ccA}(\mu,\nu)$, we have
\begin{equation}\label{eq1}
(\delta\circ_0\gamma)\circ_1(\beta\circ_0\alpha)=(\delta\circ_1\beta)\circ_0(\gamma\circ_1\alpha).
\end{equation}
Assume first that 
$\sigma=\mu=\mathrm{id}_{\mathbf{n}}$. In this case both $\alpha$, $\beta$, $\gamma$ and $\delta$
are subrelations of the identity relation $\mathrm{id}_{\mathbf{n}}$. Note that, given two subrelations
$x$ and $y$ of the identity relation $\mathrm{id}_{\mathbf{n}}$, their product $xy$ as binary relations
equals $x\cap y$. Hence, in this particular case, both sides of \eqref{eq1} are equal to 
$\alpha\cap\beta\cap\gamma\cap\delta$.  

Before proving the general case, we will need the following two lemmata:

\begin{lemma}\label{lem2}
Let $\pi,\sigma,\tau,\rho\in S_n$. 
\begin{enumerate}[$($i$)$]
\item\label{lem2.1} Left composition with $\pi$ induces a bijection from 
$\mathrm{Hom}_{\ccA}(\sigma,\tau)$ to $\mathrm{Hom}_{\ccA}(\pi\sigma,\pi\tau)$.
\item\label{lem2.2} For any $\alpha\in \mathrm{Hom}_{\ccA}(\sigma,\tau)$ and
$\beta\in \mathrm{Hom}_{\ccA}(\tau,\rho)$, we have
\begin{displaymath}
\pi\circ_0(\beta\circ_1\alpha)=(\pi\circ_0\beta)\circ_1(\pi\circ_0\alpha).
\end{displaymath}
\end{enumerate}
\end{lemma}

\begin{proof}
Left composition with $\pi$ maps an element $(y,x)$ of $\alpha\in \mathrm{Hom}_{\ccA}(\sigma,\tau)$ 
to $(\pi(y),x)\in \mathrm{Hom}_{\ccA}(\pi\sigma,\pi\tau)$. As $\pi$ is an invertible transformation 
of $\mathbf{n}$, multiplying with  $\pi^{-1}$ returns $(\pi(y),x)$ to $(y,x)$. This implies
claim~\eqref{lem2.1}. Claim~\eqref{lem2.2} follows from claim~\eqref{lem2.1} and the observation
that composition with invertible maps commutes with taking intersections.
\end{proof}

\begin{lemma}\label{lem3}
Let $\pi,\sigma,\tau,\rho\in S_n$. 
\begin{enumerate}[$($i$)$]
\item\label{lem3.1} Right composition with $\pi$ induces a bijection from 
$\mathrm{Hom}_{\ccA}(\sigma,\tau)$ to $\mathrm{Hom}_{\ccA}(\sigma\pi,\tau\pi)$.
\item\label{lem3.2} For any $\alpha\in \mathrm{Hom}_{\ccA}(\sigma,\tau)$ and
$\beta\in \mathrm{Hom}_{\ccA}(\tau,\rho)$, we have 
\begin{displaymath}
(\beta\circ_1\alpha)\circ_0\pi=(\beta\circ_0\pi)\circ_1(\alpha\circ_0\pi). 
\end{displaymath}
\end{enumerate}
\end{lemma}

\begin{proof}
Analogous to the proof of Lemma~\ref{lem2}.
\end{proof}

Using Lemmata~\ref{lem2} and \ref{lem3} together with associativity of $\circ_0$, 
right multiplication with $\sigma^{-1}$ and left multiplication with $\mu^{-1}$ reduces
the general case of \eqref{eq1} to the case $\sigma=\mu=\mathrm{id}_{\mathbf{n}}$ considered above.
This completes the proof.
\end{proof}

\subsection{Second $2$-category extending $S_n$}\label{s2.4}

For $n\in\mathbb{N}$, consider the corresponding {\em partition semigroup} $\mathbf{P}_n$,
see \cite{Jo,Mar1,Mar2,Maz}. Elements of $\mathbf{P}_n$ are partitions of the set 
\begin{displaymath}
\overline{\mathbf{n}}:=\{1,2,\dots,n,1',2',\dots,n'\} 
\end{displaymath}
into disjoint unions of non-empty subsets, called {\em parts}. Alternatively, one can view
elements of $\mathbf{P}_n$ as equivalence relations on $\overline{\mathbf{n}}$. Multiplication 
$(\rho,\pi)\mapsto \rho\pi$ in $\mathbf{P}_n$ is 
given by the following {\em mini-max algorithm}, see \cite{Jo,Mar1,Mar2,Maz} for details:
\begin{itemize}
\item Consider $\rho$ as a partition of
$\{1',2',\dots,n',1'',2'',\dots,n''\}$ using the map $x\mapsto x'$ and $x'\mapsto x''$, for $x\in\mathbf{n}$.
\item Let $\tau$ be the minimum, with respect to inclusions, partition of 
\begin{displaymath}
\{1,2,\dots,n,1',2',\dots,n',1'',2'',\dots,n''\},
\end{displaymath}
such that each part of both $\rho$ and $\pi$ is a subset of a part of $\tau$.
\item Let $\sigma$ be the maximum, with respect to inclusions, partition of 
\begin{displaymath}
\{1,2,\dots,n,1'',2'',\dots,n''\}, 
\end{displaymath}
such that each part of $\sigma$ is a subset of a part of $\tau$.
\item Define the product $\rho\pi$ as the partition of  $\mathbf{n}$ induced from $\sigma$
via the map $x''\mapsto x'$, for $x\in \mathbf{n}$.
\end{itemize}
Note that $S_n$ is, naturally, a submonoid of $\mathbf{P}_n$. Moreover, $S_n$ is the maximal 
subgroup of all invertible elements in $\mathbf{P}_n$.

A part of $\rho\in \mathbf{P}_n$ is called a {\em propagating} part provided that it intersects both sets
$\{1,2,\dots,n\}$ and $\{1',2',\dots,n'\}$. Partitions in which all parts are propagating are called
{\em propagating partitions}. The set of all  propagating partitions in $\mathbf{P}_n$ is denoted
by $\mathbf{PP}_n$, it is a submonoid of $\mathbf{P}_n$.

The monoid $\mathbf{P}_n$ is naturally ordered with respect to refinement: $\rho\leq \tau$ provided
that each part of $\rho$ is a subset of a part in $\tau$. With respect to this order, the partition
of $\mathbf{n}$ with just one part is the maximum element, while the partition of $\overline{\mathbf{n}}$ into
singletons is the minimum element. This order restricts to $\mathbf{PP}_n$. As elements of 
$\mathbf{P}_n$ are just equivalence relations, the poset $\mathbf{P}_n$ is a lattice and we 
denote by $\wedge$ and $\vee$ the corresponding meet and join operations, respectively. The poset
$\mathbf{PP}_n$ is a sublattice in $\mathbf{P}_n$ with the same meet and join. As 
$S_n\subset \mathbf{PP}_n$, all meets and joins in $\mathbf{P}_n$ of elements from  $S_n$ belong
to $\mathbf{PP}_n$.

We now define a $2$-category $\cB=\cB_n$. Similarly to Subsection~\ref{s2.3}, we start by declaring that 
\begin{itemize}
\item $\cB$ has one object $\mathtt{i}$;
\item $1$-morphisms in $\cB$ are elements in $S_n$;
\item composition of $1$-morphisms is induced from $S_n$;
\item the identity $1$-morphism is the identity transformation $\mathrm{id}_{\mathbf{n}}$.
\end{itemize}
It remains to define $2$-morphisms in $\cB$ and their compositions.
\begin{itemize}
\item For $\pi,\sigma\in S_n$, we define $\mathrm{Hom}_{\ccB}(\pi,\sigma)$ 
as the set of all $\alpha\in \mathbf{PP}_n$ such that we have both, $\pi\leq \alpha$ and $\sigma\leq \alpha$.
\item For $\pi,\sigma,\tau\in S_n$, and also any $\alpha\in\mathrm{Hom}_{\ccB}(\pi,\sigma)$ and
$\beta\in \mathrm{Hom}_{\ccB}(\sigma,\tau)$, we define $\beta\circ_1\alpha:=\beta\vee \alpha$.
\item For $\pi\in S_n$, we define the identity element in 
$\mathrm{Hom}_{\ccB}(\pi,\pi)$ to be $\pi$.
\item For $\pi,\sigma,\tau,\rho\in S_n$, and also any $\alpha\in\mathrm{Hom}_{\ccB}(\pi,\sigma)$ and
$\beta\in\mathrm{Hom}_{\ccB}(\tau,\rho)$, we define $\beta\circ_0\alpha:=\beta\alpha$, the
usual composition of partitions.
\end{itemize}

\begin{proposition}\label{prop4}
The construct $\cB$  above is a $2$-category.
\end{proposition}

\begin{proof}
The vertical composition  $\circ_1$  is clearly well-defined. It is associative as $\vee$ is associative.
If $\alpha\in\mathrm{Hom}_{\ccB}(\pi,\sigma)$ or $\alpha\in\mathrm{Hom}_{\ccB}(\sigma,\pi)$, then 
$\pi\leq \alpha$ and thus $\alpha\vee\pi =\alpha$. Therefore $\pi\in \mathrm{Hom}_{\ccB}(\pi,\pi)$ 
is the identity element.
 
Let us check that the horizontal composition $\circ_0$ is well-defined. From $\pi\leq \alpha$ and 
$\tau\leq \beta$ and the fact that $\mathbf{P}_n$ is ordered, we have $\tau\pi\leq \tau\alpha\leq \beta\alpha$.
Similarly, from $\sigma\leq \alpha$ and $\rho\leq \beta$ and the fact that $\mathbf{P}_n$ is ordered, 
we have $\rho\sigma\leq\rho\alpha\leq \beta\alpha$. It follows that 
$\beta\alpha\in \mathrm{Hom}_{\ccB}(\tau\pi,\rho\sigma)$ and thus $\circ_0$  is well-defined. 
Its associativity follows from the fact that usual composition of partitions is associative.

It remains to check the interchange law \eqref{eq1}. For this we fix any $1$-morphisms
$\pi,$ $\sigma,$ $\rho,$ $\tau,$ $\mu,$ $\nu$ and any 
$\alpha\in \mathrm{Hom}_{\ccB}(\pi,\sigma)$, $\beta\in \mathrm{Hom}_{\ccB}(\tau,\mu)$,
$\gamma\in \mathrm{Hom}_{\ccB}(\sigma,\rho)$ and $\delta\in \mathrm{Hom}_{\ccB}(\mu,\nu)$. Assume first that 
$\sigma=\mu=\mathrm{id}_{\mathbf{n}}$. In this case both $\alpha$, $\beta$, $\gamma$ and $\delta$
are partitions containing the identity relation $\mathrm{id}_{\mathbf{n}}$. Note that, given two 
partitions $x$ and $y$ containing the identity relation $\mathrm{id}_{\mathbf{n}}$, their product 
$xy$ as partitions equals $x\vee y$. Hence, in this particular case, both sides of 
\eqref{eq1} are equal to  $\alpha\vee\beta\vee\gamma\vee\delta$.  

Before proving the general case, we will need the following two lemmata:

\begin{lemma}\label{lem5}
Let $\pi,\sigma,\tau,\rho\in S_n$. 
\begin{enumerate}[$($i$)$]
\item\label{lem5.1} Left composition with $\pi$ induces a bijection between the sets 
$\mathrm{Hom}_{\ccB}(\sigma,\tau)$ and $\mathrm{Hom}_{\ccB}(\pi\sigma,\pi\tau)$.
\item\label{lem5.2} For any $\alpha\in \mathrm{Hom}_{\ccB}(\sigma,\tau)$ and
$\beta\in \mathrm{Hom}_{\ccB}(\tau,\rho)$, we have 
\begin{displaymath}
\pi\circ_0(\beta\circ_1\alpha)=(\pi\circ_0\beta)\circ_1(\pi\circ_0\alpha).
\end{displaymath}
\end{enumerate}
\end{lemma}

\begin{proof}
Left composition with $\pi$ simply renames elements of $\{1',2',\dots,n'\}$ in an
invertible way. This implies claim~\eqref{lem2.1}. Claim~\eqref{lem2.2} follows 
from claim~\eqref{lem2.1} and the observation that composition with invertible maps 
commutes with taking unions.
\end{proof}

\begin{lemma}\label{lem6}
Let $\pi,\sigma,\tau,\rho\in S_n$. 
\begin{enumerate}[$($i$)$]
\item\label{lem6.1} Right composition with $\pi$ induces a bijection between the sets 
$\mathrm{Hom}_{\ccB}(\sigma,\tau)$ and $\mathrm{Hom}_{\ccB}(\sigma\pi,\tau\pi)$.
\item\label{lem6.2} For any $\alpha\in \mathrm{Hom}_{\ccB}(\sigma,\tau)$ and
$\beta\in \mathrm{Hom}_{\ccB}(\tau,\rho)$, we have 
\begin{displaymath}
(\beta\circ_1\alpha)\circ_0\pi=(\beta\circ_0\pi)\circ_1(\alpha\circ_0\pi). 
\end{displaymath}
\end{enumerate}
\end{lemma}

\begin{proof}
Analogous to the proof of Lemma~\ref{lem5}.
\end{proof}

Using Lemmata~\ref{lem5} and \ref{lem6} together with associativity of $\circ_0$, 
right multiplication with $\sigma^{-1}$ and left multiplication with $\mu^{-1}$ reduces
the general case of \eqref{eq1} to the case $\sigma=\mu=\mathrm{id}_{\mathbf{n}}$ considered above.
This completes the proof.
\end{proof}

\section{$2$-categories in the linear world}\label{s3}

For more details on all the definitions in Section~\ref{s3}, we refer to \cite{GrMa}.

\subsection{Finitary $2$-categories}\label{s3.1}

Let $\Bbbk$ be a field. A $\Bbbk$-linear category $\mathcal{C}$ is called {\em finitary} provided that
it is additive, idempotent split and Krull-Schmidt 
(cf. \cite[Section~2.2]{Ri}) with finitely many isomorphism classes of indecomposable
objects and finite dimensional homomorphism spaces.

A $2$-category $\cC$ is called {\em prefinitary (over $\Bbbk$)} provided that
\begin{enumerate}[$($I$)$]
\item\label{cond.a} $\cC$ has finitely many objects;
\item\label{cond.b} each $\cC(\mathtt{i},\mathtt{j})$ is a finitary $\Bbbk$-linear category;
\item\label{cond.c} all compositions are biadditive and $\Bbbk$-linear whenever the notion makes sense.
\end{enumerate}

Following \cite{MM1}, a prefinitary $2$-category $\cC$ is called {\em finitary} provided that 
\begin{enumerate}[$($I$)$]
\setcounter{enumi}{3}
\item\label{cond.d} all identity $1$-morphisms are indecomposable.
\end{enumerate}

\subsection{$\Bbbk$-linearization of finite categories}\label{s3.2}

For any set $X$, let us denote by $\Bbbk[X]$ the vector space (over $\Bbbk$) of all 
formal linear combinations of elements in $X$ with coefficients in $\Bbbk$. 
Then we can view $X$ as the {\em standard basis} in $\Bbbk[X]$.
By convention, $\Bbbk[X]=\{0\}$ if $X=\varnothing$.

Let $\mathcal{C}$ be a finite category, that is a category with a finite number of  morphisms. 
Define the {\em $\Bbbk$-linearization} $\mathcal{C}_{\Bbbk}$ of
$\mathcal{C}$ as follows:
\begin{itemize}
\item the objects in $\mathcal{C}_{\Bbbk}$ and $\mathcal{C}$ are the same;
\item we have $\mathcal{C}_{\Bbbk}(\mathtt{i},\mathtt{j}):=\Bbbk[\mathcal{C}(\mathtt{i},\mathtt{j})]$;
\item composition in $\mathcal{C}_{\Bbbk}$ is induced from that in $\mathcal{C}$ by $\Bbbk$-bilinearity.
\end{itemize}

\subsection{$\Bbbk$-additivization of finite categories}\label{s3.3}

Assume that objects of the category $\mathcal{C}$ are $\mathtt{1}$, $\mathtt{2}$,\dots, $\mathtt{k}$. 
For $\mathcal{C}$ as in Subsection~\ref{s3.2}, define
the {\em additive $\Bbbk$-linearization} $\mathcal{C}_{\Bbbk}^{\oplus}$ of $\mathcal{C}$ in the following way:
\begin{itemize}
\item objects in $\mathcal{C}_{\Bbbk}^{\oplus}$ are elements in $\mathbb{Z}_{\geq 0}^k$,
we identify  $(m_1,m_2,\dots,m_k)\in \mathbb{Z}_{\geq 0}^k$ with the symbol
\begin{displaymath}
\underbrace{\mathtt{1}\oplus\dots\oplus\mathtt{1}}_{m_1 \text{ times}}\oplus 
\underbrace{\mathtt{2}\oplus\dots\oplus\mathtt{2}}_{m_2 \text{ times}}\oplus\dots\oplus 
\underbrace{\mathtt{k}\oplus\dots\oplus\mathtt{k}}_{m_k \text{ times}};
\end{displaymath}
\item the set $\mathcal{C}_{\Bbbk}^{\oplus}(\mathtt{i}_1\oplus
\mathtt{i}_2\oplus\dots\oplus\mathtt{i}_l,\mathtt{j}_1\oplus
\mathtt{j}_2\oplus\dots\oplus\mathtt{j}_m)$ is given by the set of all matrices of the form
\begin{displaymath}
\left(\begin{array}{cccc}
f_{11}& f_{12}&\dots& f_{1l} \\
f_{21}& f_{22}&\dots& f_{2l} \\
\vdots& \vdots&\ddots& \vdots \\
f_{m1}& f_{m2}&\dots& f_{ml}  
\end{array}\right) 
\end{displaymath}
where $f_{st}\in\mathcal{C}_{\Bbbk}(\mathtt{i}_t,\mathtt{j}_s)$;
\item composition in $\mathcal{C}_{\Bbbk}^{\oplus}$ is given by the usual matrix multiplication;
\item the additive structure is given by addition in $\mathbb{Z}_{\geq 0}^k$.
\end{itemize}
One should think of $\mathcal{C}_{\Bbbk}^{\oplus}$ as the additive category generated by $\mathcal{C}_{\Bbbk}$.

\subsection{$\Bbbk$-linearization of finite $2$-categories}\label{s3.4}

Let now $\cC$ be a finite $2$-category. We define the {\em $\Bbbk$-linearization} $\cC_{\Bbbk}$
of $\cC$ over $\Bbbk$ as follows:
\begin{itemize}
\item $\cC_{\Bbbk}$ and $\cC$ have the same objects;
\item we have $\cC_{\Bbbk}(\mathtt{i},\mathtt{j}):=\cC(\mathtt{i},\mathtt{j})_{\Bbbk}^{\oplus}$;
\item composition in $\cC_{\Bbbk}$ is induced from composition in $\cC$ by biadditivity and $\Bbbk$-bilinearity.
\end{itemize}
By construction, the $2$-category $\cC_{\Bbbk}$ satisfies conditions \eqref{cond.a} and 
\eqref{cond.c} from the definition of a finitary $2$-category. A part of condition 
\eqref{cond.b} related to additivity and finite dimensionality of morphism spaces is also satisfied. 
Therefore, the $2$-category $\cC_{\Bbbk}$ is finitary if and only if, the $2$-endomorphism $\Bbbk$-algebra
of every $1$-morphism in $\cC_{\Bbbk}$ is local.

\subsection{$\Bbbk$-finitarization of finite $2$-categories}\label{s3.45}

Let $\cC$ be a finite $2$-category. Consider the $2$-category $\cC_{\Bbbk}$. 
We define the {\em finitarization} $\Bbbk\cC$ of $\cC_{\Bbbk}$ as follows:
\begin{itemize}
\item $\Bbbk\cC$ and $\cC_{\Bbbk}$ have the same objects;
\item $\Bbbk\cC(\mathtt{i},\mathtt{j})$ is defined to be the idempotent completion  
of $\cC_{\Bbbk}(\mathtt{i},\mathtt{j})$;
\item composition in $\Bbbk\cC$ is induced from composition in $\cC$.
\end{itemize}
By construction, the $2$-category $\Bbbk\cC$ is prefinitary.  Therefore, the $2$-category 
$\Bbbk\cC$ is finitary if and only if, the $2$-endomorphism $\Bbbk$-algebra of every 
identity $1$-morphism in $\Bbbk\cC$ is local.

\subsection{Idempotent splitting}\label{s3.5}

Let $\cC$ be a prefinitary $2$-category.
If $\cC$ does not satisfy condition \eqref{cond.d}, then there is an object $\mathtt{i}\in\cC$ such that 
the endomorphism algebra $\mathrm{End}_{\Bbbk\ccC}(\mathbbm{1}_{\mathtt{i}})$ is not local, that 
is, contains a non-trivial idempotent. In this subsection we describe a version of ``idempotent splitting'',
for all $\mathrm{End}_{\Bbbk\ccC}(\mathbbm{1}_{\mathtt{i}})$, to turn $\cC$ into a finitary 
$2$-category which we denote by $\overline{\cC}$.

For $\mathtt{i}\in\cC$, the $2$-endomorphism algebra of  $\mathbbm{1}_{\mathtt{i}}$ is equipped with two
unital associative operations, namely, $\circ_0$ and $\circ_1$. These two operations 
satisfy the interchange law. By the classical
Eckmann-Hilton argument (see, for example, \cite{EH} or \cite[Subsection~1.1]{Ko}), both these operations,
when restricted to the $2$-endomorphism algebra of  $\mathbbm{1}_{\mathtt{i}}$, must be commutative and,
in fact, coincide. Therefore we can unambiguously speak about the commutative $2$-endomorphism algebra 
$\mathrm{End}_{\ccC}(\mathbbm{1}_{\mathtt{i}})$. Let $\varepsilon_{\mathtt{i}}^{(j)}$, where
$j=1,2,\dots,k_{\mathtt{i}}$, be a complete list of primitive idempotents in 
$\mathrm{End}_{\ccC}(\mathbbm{1}_{\mathtt{i}})$. Note that the elements $\varepsilon_{\mathtt{i}}^{(j)}$
are identities in the minimal ideals of $\mathrm{End}_{\ccC}(\mathbbm{1}_{\mathtt{i}})$ and hence 
are canonically determined (up to permutation).

We now define a new $2$-category, which we denote by $\overline{\cC}$, in the following way:
\begin{itemize}
\item Objects in $\overline{\cC}$ are $\mathtt{i}^{(s)}$, where $\mathtt{i}\in\cC$ and
$s=1,2,\dots,k_{\mathtt{i}}$.
\item $1$-morphisms in $\overline{\cC}(\mathtt{i}^{(s)},\mathtt{j}^{(t)})$ are the same as 
$1$-morphisms in ${\cC}(\mathtt{i},\mathtt{j})$.
\item for $1$-morphisms $\mathrm{F},\mathrm{G}\in \overline{\cC}(\mathtt{i}^{(s)},\mathtt{j}^{(t)})$,
the set $\mathrm{Hom}_{\overline{\ccC}}(\mathrm{F},\mathrm{G})$ equals
\begin{displaymath}
\varepsilon_{\mathtt{j}}^{(t)}\circ_0
\mathrm{Hom}_{{\ccC}}(\mathrm{F},\mathrm{G})\circ_0 \varepsilon_{\mathtt{i}}^{(s)}. 
\end{displaymath}
\item The identity $1$-morphism in $\overline{\cC}(\mathtt{i}^{(s)},\mathtt{i}^{(s)})$
is $\mathbbm{1}_{\mathtt{i}}$.
\item All compositions are induced from $\cC$.
\end{itemize}

\begin{lemma}\label{lem91}
Let $\cC$ be a prefinitary $2$-category. Then
the construct $\overline{\cC}$ is a finitary $2$-category. 
\end{lemma}

\begin{proof}
The fact that $\overline{\cC}$ is a $2$-category follows from the fact that 
$\cC$ is a $2$-category, by construction. For $\overline{\cC}$, conditions \eqref{cond.a}, 
\eqref{cond.b} and \eqref{cond.c} from the definition of  a prefinitary $2$-category,
follow from the corresponding conditions for the original category $\cC$.

It remains to show that $\overline{\cC}$ satisfies \eqref{cond.d}. By construction, the
endomorphism algebra of the identity $1$-morphism $\mathbbm{1}_{\mathtt{i}}$ in 
$\overline{\cC}(\mathtt{i}^{(s)},\mathtt{i}^{(s)})$ is
\begin{displaymath}
\varepsilon_{\mathtt{i}}^{(s)}\circ_0\mathrm{End}_{\ccC}(\mathbbm{1}_{\mathtt{i}})\circ_0 
\varepsilon_{\mathtt{i}}^{(s)}.
\end{displaymath}
The latter algebra is local as $\varepsilon_{\mathtt{i}}^{(s)}$ is a minimal idempotent.
This means that condition \eqref{cond.d} is satisfied and completes the proof.
\end{proof}

Starting from $\overline{\cC}$ and taking, for each $\mathtt{i}\in\cC$, a direct sum of 
$\mathtt{i}^{(s)}$, where $s=1,2,\dots,k_{\mathtt{i}}$, one obtains a $2$-category biequivalent
to the original $2$-category $\cC$. The $2$-categories $\cC$ and $\overline{\cC}$ are, clearly, 
Morita equivalent in the sense of \cite{MM4}. 

{\bf Warning:} Despite of the fact that $\overline{\cC}(\mathtt{i}^{(s)},\mathtt{j}^{(t)})$
and ${\cC}(\mathtt{i},\mathtt{j})$ have {\em the same} $1$-morphisms, these two categories, in general,
have {\em different}  indecomposable $1$-mor\-phisms as the sets of $2$-morphisms are different.
In particular, indecomposable $1$-morphisms in ${\cC}(\mathtt{i},\mathtt{j})$ may become isomorphic to 
zero in $\overline{\cC}(\mathtt{i}^{(s)},\mathtt{j}^{(t)})$.

We note that the operation of idempotent splitting is also known as taking
{\em Cauchy completion} or {\em Karoubi envelope}.

\section{Comparison of $\overline{\Bbbk\cA_n}$ and $\overline{\Bbbk\cB_n}$ 
to $2$-categories associated with ordered monoids $IS_n$ and $F^*_n$}\label{s4}

\subsection{$2$-categories and ordered monoids}\label{s4.0}

Let $(S,\cdot,1)$ be a monoid and $\leq $ be an {\em admissible order} on $S$, that is a partial (reflexive) 
order such that $s\leq t$ implies both $sx\leq tx$ and $xs\leq xt$, for all $x,s,t\in S$. 
Then we can associate with $S$ a $2$-category $\cS=\cS_S=\cS_{(S,\cdot,1,\leq)}$ defined as follows:
\begin{itemize}
\item $\cS$ has one object $\mathtt{i}$;
\item $1$-morphisms are elements in $S$;
\item for $s,t\in S$, the set $\mathrm{Hom}_{\ccS}(s,t)$ is empty if $s\not\leq t$ and contains one element
$(s,t)$ otherwise;
\item composition of $1$-morphisms is given by $\cdot$;
\item both horizontal and vertical compositions of $2$-morphism are the only possible compositions
(as sets of $2$-morphisms are either empty or singletons);
\item the identity $1$-morphism is $1$.
\end{itemize}
Admissibility of $\leq$ makes the above well-defined and ensures that $\cS$ becomes a $2$-category.

A canonical example of the above is when $S$ is an inverse monoid and $\leq $ is the {\em natural partial order}
on $S$ defined as follows: $s\leq t$ if and only if $s=e t$ for some idempotent $e\in S$.

\subsection{(Co)ideals of ordered semigroups}\label{s4.1}

Let $S$ be a semigroup equipped with an admissible order $\leq$. For a non-empty subset $X\subset S$, let
\begin{displaymath}
X^{\downarrow}:=\{s\in S\,:\, \text{ there is }x\in X\text{ such that }s\leq x\} 
\end{displaymath}
denote the {\em lower set} or {\em ideal} generated by $X$. Let
\begin{displaymath}
X^{\uparrow}:=\{s\in S\,:\, \text{ there is }x\in X\text{ such that }x\leq s\} 
\end{displaymath}
denote the {\em upper set} or {\em coideal} generated by $X$.

\begin{lemma}\label{lem7}
For any subsemigroup $T\subset S$, both $T^{\downarrow}$ and $T^{\uparrow}$ are subsemigroups of $S$.
\end{lemma}

\begin{proof}
We prove the claim for $T^{\downarrow}$, for $T^{\uparrow}$ the arguments are similar. Let $a,b\in T^{\downarrow}$.
Then there exist $s,t\in T$ such that $a\leq s$ and $b\leq t$. As $\leq$ is admissible, we have
$ab\leq sb\leq st$. Now, $st\in T$ as $T$ is a subsemigroup, and thus $ab\in T^{\downarrow}$.
\end{proof}

\subsection{The symmetric inverse monoid}\label{s4.2}

For $n\in\mathbb{N}$, we denote by $IS_n$ the {\em symmetric inverse monoid} on $\mathbf{n}$, see \cite{GM}.
It consists of all bijections between subsets of $\mathbf{n}$. Alternatively, we can identify
$IS_n$ with $S_n^{\downarrow}$ inside the ordered monoid $\mathbf{B}_n$. 
The monoid $IS_n$ is an inverse monoid. The natural partial 
order on the inverse monoid $IS_n$ coincides with the inclusion order inherited from $\mathbf{B}_n$.
The group $S_n$ is the group of invertible elements in $IS_n$.

\subsection{The dual symmetric inverse monoid}\label{s4.3}

For $n\in\mathbb{N}$, we denote by $I^*_n$ the {\em dual symmetric inverse monoid} on $\mathbf{n}$, see \cite{FL}.
It consists of all bijections between quotients of $\mathbf{n}$. Alternatively, we can identity $I^*_n$ 
with $\mathbf{PP}_n$ in the obvious way. The monoid $I^*_n$ is an inverse monoid. The natural partial 
order on the inverse monoid $I^*_n$ coincides with the order inherited from $\mathbf{P}_n$.
The group $S_n$ is the group of invertible elements in $I^*_n$. 

We also consider the {\em maximal factorizable submonoid} $F^*_n$ of $I^*_n$, that is the submonoid of all 
elements which can be written in the form $\sigma\varepsilon$, where $\sigma\in S_n$ and 
$\varepsilon$ is an idempotent in $I^*_n$. Idempotents in $I^*_n$ are exactly the identity transformations
of quotient sets of  $\mathbf{n}$, equivalently, idempotents in $I^*_n$ coincide with the principal
coideal in $\mathbf{P}_n$ generated by the identity element.

\begin{lemma}\label{lem8}
The monoid $F^*_n$ coincides with the subsemigroup $S_n^{\uparrow}$ of $\mathbf{PP}_n$.
\end{lemma}

\begin{proof}
As $S_n^{\uparrow}$ contains both $S_n$ and all idempotents of  $I^*_n$, we have 
$F^*_n\subset S_n^{\uparrow}$. On the other hand, let $\rho\in S_n^{\uparrow}$. Then
$\sigma\leq \rho$ for some $\sigma\in S_n$. This means that $\mathrm{id}_{\mathbf{n}}\leq \rho\sigma^{-1}$.
Hence $\rho\sigma^{-1}$ is an idempotent and $\rho=(\rho\sigma^{-1})\sigma\in F^*_n$.
\end{proof}

\subsection{Fiat $2$-categories}\label{s4.4}

Following \cite{MM1}, we say that a finitary $2$-category $\cC$ is {\em fiat} provided that there exists
a weak anti-involution $\star:\cC\to\cC^{\mathrm{co},\mathrm{op}}$, 
such that, for any objects $\mathtt{i},\mathtt{j}\in\cC$ and any
$1$-morphism $\mathrm{F}\in \cC(\mathtt{i},\mathtt{j})$, there 
are $2$-morphisms $\eta:\mathbbm{1}_{\mathtt{i}}\to \mathrm{F}^{\star}\mathrm{F}$ and
$\varepsilon: \mathrm{F}\mathrm{F}^{\star}\to \mathbbm{1}_{\mathtt{j}}$ such that 
\begin{displaymath}
(\varepsilon\circ_0 \mathrm{id}_{\mathrm{F}})\circ_1(\mathrm{id}_{\mathrm{F}}\circ_0\eta)=\mathrm{id}_{\mathrm{F}} 
\quad\text{ and }\quad
(\mathrm{id}_{\mathrm{F}^{\star}}\circ_0 \varepsilon)\circ_1
(\eta\circ_0\mathrm{id}_{\mathrm{F}^{\star}})=\mathrm{id}_{\mathrm{F}^{\star}}.
\end{displaymath}
This means that $\mathrm{F}$ and $\mathrm{F}^{\star}$ are biadjoint in $\cC$ and hence also in
any $2$-representation of $\cC$. 
The above property is usually called {\em existence of adjunction $2$-morphisms}.

There are various classes of $2$-categories whose axiomatization covers some parts
of the axiomatization of fiat $2$-categories, see, for example, 
{\em compact categories}, {\em rigid categories}, {\em monoidal categories with duals}
and {\em 2-categories with adjoints}.

\subsection{Comparison of fiatness}\label{s4.6}

\begin{theorem}\label{thmmain}
Let $n\in\mathbb{N}$.

\begin{enumerate}[$($i$)$]
\item\label{thmmain.1} Both $2$-categories, $\overline{\Bbbk\cA_n}$ and $\overline{\Bbbk\cB_n}$, 
are fiat.
\item\label{thmmain.2} Both $2$-categories, $\Bbbk\cS_{IS_n}$ and $\Bbbk\cS_{F_n^*}$, are finitary but not fiat.
\end{enumerate}
\end{theorem}

\begin{proof}
The endomorphism algebra of any $1$-morphism in $\Bbbk\cS_{IS_n}$ is $\Bbbk$, by definition.
Therefore $\Bbbk\cS_{IS_n}$ is finitary by construction.
The category $\Bbbk\cS_{IS_n}$ cannot be fiat as it contains non-invertible indecomposable $1$-morphisms 
but it does not contain any non-zero $2$-morphisms from the identity $1$-morphism to any non-invertible
indecomposable $1$-morphism. Therefore adjunction $2$-morphisms for non-invertible
indecomposable $1$-morphisms cannot exist. The same argument also applies to $\Bbbk\cS_{F_n^*}$,
proving claim~\eqref{thmmain.2}.
 
By construction, the $2$-category $\Bbbk\cA_n$ satisfies conditions \eqref{cond.a}, \eqref{cond.b} and 
\eqref{cond.c} from the definition of a finitary $2$-category. Therefore the $2$-category 
$\overline{\Bbbk\cA_n}$ is a finitary $2$-category by Lemma~\ref{lem91}. Let us now check existence 
of adjunction $2$-morphisms. 

Recall that an adjoint to a direct sum of functors is a direct sum of adjoints to components.
Therefore, as $\overline{\Bbbk\cA_n}$ is obtained from $(\cA_n)_{\Bbbk}$ by splitting idempotents in 
$2$-endomorphism rings, it is enough to check that adjunction $2$-morphisms exist in
$(\cA_n)_{\Bbbk}$. Any $1$-morphism in $(\cA_n)_{\Bbbk}$ is, by construction, a direct sum of
$\sigma\in S_n$. Therefore it is enough to check that adjunction $2$-morphisms exist in
$\cA_n$.  In the latter category, each $1$-morphism $\sigma\in S_n$ is invertible and hence 
both left and right adjoint to $\sigma^{-1}$. This implies existence of adjunction $2$-morphisms 
in $\cA_n$. 

The above shows that the $2$-category $\overline{\Bbbk\cA_n}$ is fiat. Similarly one shows that 
the $2$-category $\overline{\Bbbk\cB_n}$ is fiat. This completes the proof.
\end{proof}

\section{Decategorification}\label{s5}

\subsection{Decategorification via Grothendieck group}\label{s5.1}

Let $\cC$ be a finitary $2$-category. A {\em Grothendieck decategorification} $[\cC]$ of $\cC$ is a category 
defined as follows:
\begin{itemize}
\item $[\cC]$ has the same objects as $\cC$.
\item For $\mathtt{i},\mathtt{j}\in \cC$, the set $[\cC](\mathtt{i},\mathtt{j})$ coincides with the
split Grothendieck group $[\cC(\mathtt{i},\mathtt{j})]_{\oplus}$ of the additive category 
$\cC(\mathtt{i},\mathtt{j})$.
\item The identity morphism in $[\cC](\mathtt{i},\mathtt{i})$ is the class of $\mathbbm{1}_{\mathtt{i}}$.
\item Composition in $[\cC]$ is induced from composition of $1$-morphisms in $\cC$.
\end{itemize}
We refer to \cite[Lecture~1]{Maz2} for more details.

For a finitary $2$-category $\cC$, the above allows us to define the {\em decategorification} of 
$\cC$ as the $\mathbb{Z}$-algebra 
\begin{displaymath}
A_{\ccC}:=\bigoplus_{\mathtt{i},\mathtt{j}\in\ccC} [\cC](\mathtt{i},\mathtt{j})
\end{displaymath}
with the induced composition. The algebra $A_{\ccC}$ is positively based in the sense of \cite{KiMa2} with
respect to the basis corresponding to indecomposable $1$-morphisms in $\cC$.

\subsection{Decategorifications of $\overline{\Bbbk\cA_n}$ and $\Bbbk\cS_{IS_n}$}\label{s5.2}

\begin{theorem}\label{thm5}
We have $A_{\overline{\Bbbk\ccA_n}}\cong A_{\Bbbk\ccS_{IS_n}}\cong\mathbb{Z}[IS_n]$.
\end{theorem}

\begin{proof}
Indecomposable $1$-morphisms in $\Bbbk\cS_{IS_n}$ correspond exactly to elements of $IS_n$, by construction.
This implies that $A_{\Bbbk\ccS_{IS_n}}\cong\mathbb{Z}[IS_n]$ where an indecomposable $1$-morphism
$\sigma$ on the left hand side is mapped to itself on the right hand side.
So, we only need to prove that $A_{\overline{\Bbbk\ccA_n}}\cong\mathbb{Z}[IS_n]$.

For $\sigma\in IS_n$, set
\begin{equation}\label{eq91}
\underline{\sigma} :=\sum_{\rho\subset \sigma}(-1)^{|\sigma\setminus\rho|}\rho\in\mathbb{Z}[IS_n].
\end{equation}
Then $\{\underline{\sigma}\,:\,\sigma\in IS_n\}$ is a basis in $\mathbb{Z}[IS_n]$ which we call the 
{\em M{\"o}bius basis}, see, for example, \cite[Theorem~4.4]{Ste}.

The endomorphism monoid $\mathrm{End}_{\ccA_n}(\mathrm{id}_{\mathbf{n}})$ is, by construction, canonically
isomorphic to the Boolean $2^{\mathbf{n}}$ of $\mathbf{n}$ with both $\circ_0$ and $\circ_1$ being equal to the
operation on $2^{\mathbf{n}}$ of taking the intersection. We identify elements in 
$\mathrm{End}_{\ccA_n}(\mathrm{id}_{\mathbf{n}})$
and in $2^{\mathbf{n}}$ in the obvious way. With this identification, in the construction of 
$\overline{\Bbbk\cA_n}$, we can take, for $X\subset \mathbf{n}$,
\begin{equation}\label{eq92}
\varepsilon^{(X)}_{\mathtt{i}}=\sum_{Y\subseteq X}(-1)^{|X|-|Y|}Y.
\end{equation}

For $\sigma\in S_n$ and $X,Y\subset \mathbf{n}$, consider the element
\begin{equation}\label{eq93}
\varepsilon^{(Y)}_{\mathtt{i}}\circ_0\sigma\circ_0\varepsilon^{(X)}_{\mathtt{i}}\in
\mathrm{End}_{\Bbbk\cccA_n}(\sigma)
\end{equation}
and write it as a linear combination of subrelations of $\sigma$ (this is the standard basis
in $\mathrm{End}_{\Bbbk\cccA_n}(\sigma)$). A subrelation $\rho\subset \sigma$
may appear in this linear combination with a non-zero coefficient only if $\rho$ consist of pairs
of the form $(y,x)$, where $x\in X$ and $y\in Y$. 

Assume that $\sigma(X)=Y$. Then the relation 
\begin{displaymath}
\rho_{\sigma}=\bigcup_{x\in X}\{(\sigma(x),x)\},
\end{displaymath}
clearly, appears in the linear combination above with coefficient one. Moreover, the idempotent properties
of $\varepsilon^{(X)}_{\mathtt{i}}$ and $\varepsilon^{(Y)}_{\mathtt{i}}$ imply that the element
in \eqref{eq93} is exactly $\underline{\rho_{\sigma}}$.

Assume that $\sigma(X)\neq Y$. Then the inclusion-exclusion formula implies that any 
subrelation of $\sigma$ appears in the linear combination above with coefficient zero.
This means that the $1$-morphism $\sigma\in \overline{\Bbbk\cA_n}(\mathtt{i}^{(Y)},\mathtt{i}^{(X)})$ 
is zero if and only if $\sigma(X)\neq Y$.

If $|X|=|Y|$ and $\sigma,\pi\in S_n$ are such that $\sigma(x)=\pi(x)\in Y$, for all $x\in X$, then
\begin{displaymath}
\underline{\rho_{\sigma}}=\underline{\rho_{\pi}}\in \mathrm{Hom}_{\Bbbk\ccA_n}(\sigma,\pi)\cap 
\mathrm{Hom}_{\Bbbk\ccA_n}(\pi,\sigma)
\end{displaymath}
gives rise to an isomorphism between $\sigma$ and $\pi$ in 
$\overline{\Bbbk\cA_n}(\mathtt{i}^{(Y)},\mathtt{i}^{(X)})$. 
If $\sigma(x)\neq \pi(x)$, for some $x\in X$, then any morphism in $\mathrm{Hom}_{\Bbbk\ccA_n}(\sigma,\pi)$
is a linear combination of relations which are properly contained in both $\rho_{\sigma}$
and $\rho_{\pi}$. Therefore $\sigma$ and $\pi$ are not isomorphic in $\Bbbk\cA_n$.

Consequently, isomorphism classes of 
indecomposable $1$-morphisms in the category $\overline{\Bbbk\cA_n}(\mathtt{i}^{(Y)},\mathtt{i}^{(X)})$ correspond 
precisely to elements in  $IS_n$ with domain $X$ and image $Y$. Composition of these
indecomposable $1$-morphisms is inherited from $S_n$. By comparing formulae~\eqref{eq91} and \eqref{eq92},
we see that composition of $1$-morphisms in $\overline{\Bbbk\cA_n}$ corresponds to multiplication 
of the M{\"o}bius basis elements in $\mathbb{Z}[IS_n]$. This completes the proof of the theorem.
\end{proof}

Theorem~\ref{thm5} allows us to consider $\overline{\Bbbk\cA_n}$ and $\cS_{IS_n}$ as two different
categorifications of $IS_n$. The advantage of $\overline{\Bbbk\cA_n}$ is that this $2$-category is fiat.

The construction we use in our proof of Theorem~\ref{thm5} resembles the partialization construction 
from \cite{KuMa1}.

\subsection{Decategorifications of $\overline{\Bbbk\cB_n}$ and $\Bbbk\cS_{F_n^*}$}\label{s5.3}

\begin{theorem}\label{thm7}
We have $A_{\overline{\Bbbk\ccB_n}}\cong A_{\Bbbk\ccS_{\hspace{-1mm}F_n^*}}\cong\mathbb{Z}[F_n^*]$.
\end{theorem}

\begin{proof}
Using the M{\"o}bius function for the poset of all quotients of $\mathbf{n}$ with respect to $\leq$
(see, for example, \cite[Example~1]{Rot}),
Theorem~\ref{thm7} is proved mutatis mutandis Theorem~\ref{thm5}.
\end{proof}

Theorem~\ref{thm7} allows us to consider $\overline{\Bbbk\cB_n}$ and $\cS_{F_n^*}$ as two different
categorifications of $F_n^*$. The advantage of $\overline{\Bbbk\cB_n}$ is that this $2$-category is fiat.

The immediately following examples are in low rank, but show that these
constructions can be worked with at the concrete as well as the abstract level.
In particular, they illustrate the difference between the two constructions.

\section{Examples for $n=2$}\label{s6}

\subsection{Example of $F^*_2$}\label{s6.1}

The monoid $F^*_2$ consists of three elements which we write as follows:
\begin{displaymath}
\epsilon:=\left(\begin{array}{cc}1&2\\1&2\end{array}\right),\quad
\sigma:=\left(\begin{array}{cc}1&2\\2&1\end{array}\right),\quad
\tau:=\left(\begin{array}{c}\{1,2\}\\\{1,2\}\end{array}\right).
\end{displaymath}
These are identified with the following partitions of $\{1,2,1',2'\}$:
\begin{displaymath}
\epsilon\leftrightarrow \big\{\{1,1'\},\{2,2'\}\big\},\quad 
\sigma\leftrightarrow \big\{\{1,2'\},\{2,1'\}\big\},\quad 
\tau\leftrightarrow \big\{\{1,2,1',2'\}\big\}.
\end{displaymath}
The symmetric group $S_2$ consists of $\epsilon$ and $\sigma$.

Here is the table showing all $2$-morphisms in $\cB_2$ from $x$ to $y$:
\begin{displaymath}
\begin{array}{c||c|c|c}
y\setminus x & \epsilon &\sigma & \tau\\
\hline\hline
\epsilon& \epsilon,\tau & \tau & \tau\\
\hline
\sigma&\tau & \sigma,\tau& \tau\\
\hline
\tau&\tau&\tau&\tau
\end{array}
\end{displaymath}
The $2$-endomorphism algebra of both $\epsilon$ and $\sigma$ in $(\cB_2)_{\Bbbk}$ is isomorphic
to $\Bbbk\oplus\Bbbk$ where the primitive idempotents are $\tau$ and $\epsilon-\tau$,
in the case of  $\epsilon$, and $\tau$ and $\sigma-\tau$, in the case of  $\sigma$.

The $2$-category $\Bbbk\cB_2$ has three isomorphism classes of indecomposable $1$-morphisms, namely
$\tau$, $\epsilon-\tau$ and $\sigma-\tau$. 

The $2$-category $\overline{\Bbbk\cB_2}$ has two objects, $\mathtt{i}_{\tau}$ and 
$\mathtt{i}_{\epsilon-\tau}$. The indecomposable $1$-morphisms in $\Bbbk\cB_2$
give indecomposable $1$-morphisms in $\overline{\Bbbk\cB_2}$ from $x$ to $y$ as follows:
\begin{displaymath}
\begin{array}{c||c|c}
y\setminus x & \mathtt{i}_{\epsilon-\tau} &\mathtt{i}_{\tau} \\
\hline\hline
\mathtt{i}_{\epsilon-\tau}& \epsilon-\tau,\sigma-\tau & \varnothing\\
\hline
\mathtt{i}_{\tau}&\varnothing &  \tau\\
\end{array}
\end{displaymath}

\subsection{Example of $IS_2$}\label{s6.2}

We write elements of $IS_2$ as follows:
\begin{gather*}
\epsilon:=\left(\begin{array}{cc}1&2\\1&2\end{array}\right),\quad
\sigma:=\left(\begin{array}{cc}1&2\\2&1\end{array}\right),\quad
\tau:=\left(\begin{array}{cc}1&2\\\varnothing&\varnothing\end{array}\right),\\
\alpha:=\left(\begin{array}{cc}1&2\\1&\varnothing\end{array}\right),\quad
\beta:=\left(\begin{array}{cc}1&2\\2&\varnothing\end{array}\right),\quad
\gamma:=\left(\begin{array}{cc}1&2\\\varnothing&1\end{array}\right),\quad
\delta:=\left(\begin{array}{cc}1&2\\\varnothing&2\end{array}\right).
\end{gather*}
The symmetric group $S_2$ consists of $\epsilon$ and $\sigma$.

Here is the table showing all $2$-morphisms in $\cA_2$ from $x$ to $y$:
\begin{displaymath}
\begin{array}{c||c|c|c|c|c|c|c}
y\setminus x & \epsilon &\sigma & \tau& \alpha &\beta &\gamma & \delta\\
\hline\hline
\epsilon& \epsilon,\alpha,\delta,\tau & \tau &\tau &\alpha,\tau &\tau &\tau &\delta,\tau \\
\hline
\sigma&\tau & \sigma,\beta,\gamma,\tau &\tau &\tau &\beta,\tau &\gamma,\tau &\tau \\
\hline
\tau&\tau & \tau &\tau &\tau &\tau &\tau &\tau \\
\hline
\alpha&\alpha,\tau & \tau &\tau &\alpha,\tau &\tau &\tau &\tau \\
\hline
\beta&\tau & \beta,\tau &\tau &\tau &\beta,\tau &\tau &\tau \\
\hline
\gamma&\tau & \beta,\tau &\tau &\tau &\tau &\gamma,\tau &\tau \\
\hline
\delta&\alpha,\tau & \tau &\tau &\tau &\tau &\tau &\delta,\tau \\
\hline
\end{array}
\end{displaymath}
The $2$-endomorphism algebra of  $\epsilon$ in $(\cA_2)_{\Bbbk}$ is isomorphic
to $\Bbbk\oplus\Bbbk\oplus\Bbbk\oplus\Bbbk$ where the primitive idempotents are 
$\tau$, $\alpha-\tau$, $\delta-\tau$ and $\epsilon-\alpha-\delta+\tau$.
Similarly one can describe the $2$-endomorphism algebra of  $\sigma$ in $(\cA_2)_{\Bbbk}$.
The $2$-endomorphism algebra of $\alpha$ in $(\cA_2)_{\Bbbk}$ is isomorphic
to $\Bbbk\oplus\Bbbk$ where the primitive idempotents are 
$\tau$ and $\alpha-\tau$. Similarly one can describe the $2$-endomorphism algebras of  $\beta$,
$\gamma$ and $\delta$.

The $2$-category $\Bbbk\cA_2$ has seven isomorphism classes of indecomposable $1$-morphisms, namely
\begin{displaymath}
\tau,\,\,\,  \alpha-\tau,\,\,\, \beta-\tau,\,\,\, \gamma-\tau,\,\,\, \delta-\tau,\,\,\,
\epsilon-\alpha-\delta+\tau,\,\,\,\sigma-\beta-\gamma+\tau.
\end{displaymath}

The $2$-category $\overline{\Bbbk\cA_2}$ has four objects, $\mathtt{i}_{\tau}$, 
$\mathtt{i}_{\alpha-\tau}$, $\mathtt{i}_{\delta-\tau}$
and $\mathtt{i}_{\epsilon-\alpha-\delta+\tau}$. The indecomposable $1$-morphisms in $\Bbbk\cA_2$
give indecomposable $1$-morphisms in $\overline{\Bbbk\cA_2}$ from $x$ to $y$ as follows:
\begin{displaymath}
\begin{array}{c||c|c|c|c}
y\setminus x & \mathtt{i}_{\epsilon-\alpha-\delta+\tau} &
\mathtt{i}_{\alpha-\tau}&\mathtt{i}_{\delta-\tau}&\mathtt{i}_{\tau} \\
\hline\hline
\mathtt{i}_{\epsilon-\alpha-\delta+\tau}& 
\epsilon-\alpha-\delta+\tau,\sigma-\beta-\gamma+\tau & \varnothing& \varnothing& \varnothing\\
\hline
\mathtt{i}_{\alpha-\tau}&\varnothing & \alpha-\tau& \gamma-\tau&  \varnothing\\
\hline
\mathtt{i}_{\delta-\tau}&\varnothing & \beta-\tau& \delta-\tau&  \varnothing\\
\hline
\mathtt{i}_{\tau}&\varnothing & \varnothing& \varnothing&  \tau\\
\end{array}
\end{displaymath}
This table can be compared with \cite[Figure~1]{MS}.

\vspace{0.3cm}

\noindent
P.~M.: Department of Pure Mathematics, University of Leeds, Leeds, 
LS2 9JT, UK, e-mail: {\tt ppmartin\symbol{64}maths.leeds.ac.uk }
\vspace{0.3cm}

\noindent
V.~M: Department of Mathematics, Uppsala University, Box. 480,
SE-75106, Uppsala, SWEDEN, email: {\tt mazor\symbol{64}math.uu.se}

\end{document}